\documentclass[11pt,reqno]{amsart}

\usepackage{amsmath,amsfonts,amsthm,amscd,amssymb,graphicx}
\usepackage[utf8]{inputenc}
\usepackage{amsbsy}
\usepackage{graphicx}
\usepackage{subcaption}
\usepackage{hyperref}
\usepackage[text={6.5in,9in},centering,includefoot,foot=0.6in]{geometry}

\usepackage{todonotes}

\definecolor{skyblue}{rgb}{0.85,0.85,1}

\newtheorem{prop}{Proposition}
\newtheorem{theorem}{Theorem}

\newtheorem{lemma}{Lemma}
\newtheorem{rem}{Remark}
\newtheorem{hyp}{Hypothesis}

\newcommand{\bbR}{\mathbb{R}}             
\newcommand{\bbZ}{\mathbb{Z}}             

\newcommand{\p}{\partial}				

\newcommand{\pO}{{\partial \Omega}}
\newcommand{\pOt}{{\partial \Omega_t}}

\newcommand{\Hp}{H^{1/2}(\pO)}
\newcommand{\Hm}{H^{-1/2}(\pO)}
\newcommand{\Hh}{\Hp \oplus \Hm}

\DeclareMathOperator{\Mas}{Mas}
\DeclareMathOperator{\codim}{codim}

\DeclareMathOperator{\dv}{div}

\DeclareMathOperator{\tr}{tr}

\DeclareMathOperator{\spn}{span}

\newcommand{\cF}{{\mathcal{F}}}

\newcommand{\cH}{{\mathcal{H}}}

\newcommand{\cL}{{\mathcal{L}}}

\begin{document}

\title{A symplectic perspective on constrained eigenvalue problems}
\author{Graham Cox}\email{gcox@mun.ca}\address{Department of Mathematics and Statistics, Memorial University of Newfoundland, St. John's, NL A1C 5S7, Canada}
\author{Jeremy L. Marzuola}\email{marzuola@email.unc.edu}\address{Dept. of Mathematics, UNC-CH, CB 3250 Phillips Hall, Chapel Hill, NC 27599-3250, USA}

\maketitle
\begin{abstract}
The Maslov index is a powerful tool for computing spectra of selfadjoint, elliptic boundary value problems. This is done by counting intersections of a fixed Lagrangian subspace, which designates the boundary condition, with the set of Cauchy data for the differential operator. We apply this methodology to constrained eigenvalue problems, in which the operator is restricted to a (not necessarily invariant) subspace. The Maslov index is defined and used to compute the Morse index of the constrained operator. We then prove a constrained Morse index theorem, which says that the Morse index of the constrained problem equals the number of constrained conjugate points, counted with multiplicity, and give an application to the nonlinear Schr\"odinger equation.
\end{abstract}

%
%
%

\section{Introduction}
Consider the nonlinear Schr\"odinger equation
\begin{equation}\label{eq:schr}
	-i \frac{\partial \psi}{\partial t} = \Delta \psi + f(|\psi|^2) \psi
\end{equation}
on a bounded domain $\Omega \subset \bbR^n$.
This admits a stationary solution of the form $\psi(x,t) = e^{-i \omega t} \phi(x)$ precisely when $\phi$ solves the nonlinear elliptic equation
\begin{equation}
\label{eq:statschr}
	\Delta \phi + f(\phi^2) \phi + \omega \phi = 0.
\end{equation}
The existence of nontrivial solutions to such equations on bounded domains can be seen as far back as the work of Pohozaev \cite{Poh}. See for instance \cite{CMMT} for a recent generalization to compact manifolds with boundary and a fairly complete history of the problem (though note that the results therein specify power-law nonlinearities: $f(s^2) = s^{p}$ for $1 < p < \frac{4}{d-2}$).

Assuming the existence of a solution $\phi$ to \eqref{eq:statschr}, we can then study perturbative solutions to \eqref{eq:schr} of the form $u(x,t) =  e^{-i \omega t} \big( \phi (x) + e^{\lambda t} w(x) \big)$. Plugging this ansatz into \eqref{eq:schr} and dropping higher-order terms in $w$ yields the system of eigenvalue equations
\begin{align}\label{linschr}
	L_+ u = - \lambda v, \quad L_- v = \lambda u,
\end{align}
where we have written $w = u+iv$ and $L_\pm$ are the operators
\begin{align}
	\label{lminus}
	& L_- = -\Delta - f( \phi^2) - \omega \\
	\label{lplus}
	& L_+ = - \Delta - f(\phi^2) - 2 f' (\phi^2) \phi^2 - \omega .
\end{align}

%

The eigenvalue problem \eqref{linschr} is not selfadjoint, even though $L_+$ and $L_-$ are. If $L_-$ is invertible, this system is equivalent to $L_+ u = - \lambda^2 (L_-)^{-1} u$. However, $L_-$ typically has a one-dimensional kernel generated by the bound state one is studying, since the standing wave equation \eqref{eq:statschr} is just $L_- \phi = 0$. This lack of invertibility can be overcome by restricting the problem to the subspace $(\ker L_-)^\perp \subset L^2(\Omega)$, and so one needs to describe the spectrum of the corresponding constrained $L_+$ operator. It can be shown, for instance, than unstable eigenvalues (namely those with positive real part) exist if the number of negative eigenvalues of $L_+$ constrained to $(\ker L_-)^\perp$ differs from the number of negative eigenvalues of $L_-$.  See the early work of Jones \cite{jones1988instability} and Grillakis \cite{grillakis1988linearized,grillakis1990analysis} for analysis of this statement.  For a modern treatment that captures many of the important ideas, see for instance \cite[Theorem 3.2]{kapitula2012stability}. A thorough discussion of the constrained eigenvalue problem and its role in stability theory can be found in \cite[\S 4.2]{P11} and also in \cite[\S 5.2]{KP13}, particularly Theorem 5.2.11.
  
In certain cases, for instance if $\phi$ is the positive ground state of a constrained minimization problem, the linear stability or instability can be ascertained from a constrained Morse index calculation.  In other settings for instance involving excited states, linear stability criterion are harder to establish and generally are computed numerically.  However, the nature of the such calculations can often be related to the Krein signature, which can also be framed in terms of a constrained eigenvalue problem, see \cite{kapitula2004counting, mackay1987stability}.


Motivated by the above considerations, we are thus interested in describing the spectrum, and in particular the number of negative eigenvalues, of a Schr\"odinger operator $L = -\Delta + V$ on a bounded domain $\Omega$, constrained to act on a closed subspace of $L^2(\Omega)$. In this paper we give a symplectic formulation of this problem, and use it to prove a constrained version of the celebrated Morse--Smale index theorem. We begin by reviewing the symplectic formulation of the unconstrained spectral problem, which first appeared in \cite{DJ11}, and was elaborated on in \cite{CJLS14,CJM14}.

\begin{hyp}
\label{hyp:domain}
$\Omega \subset \bbR^n$ is a bounded domain with Lipschitz boundary, and $V \in L^\infty(\Omega)$.
\end{hyp}

We define the space of Cauchy data for $L$
\begin{align}\label{mudef}
	\mu(\lambda) = \left\{ \left.\left( u, \frac{\p u}{\p \nu} \right)\right|_{\pO} : L u = \lambda u \right\},
\end{align}
where the equation $Lu = \lambda u$ is meant in a weak sense. That is, $D(u,v) = \lambda \left<u,v\right>$ for all $v \in H^1_0(\Omega)$, where $\left<\cdot,\cdot\right>$ is the $L^2$ inner product and $D$ is the bilinear form
\begin{align}\label{Ddef}
	D(u,v) = \int_\Omega \left[ \nabla u \cdot \nabla v + Vuv \right].
\end{align}
 It is known that $\mu(\lambda)$ defines a smooth curve of Lagrangian subspaces in the symplectic Hilbert space $\Hh$.
 
We let $\beta$ be a Lagrangian subspace that encodes the boundary conditions. For simplicity we take $\beta$ to be either
\begin{align}\label{betaD}
	\beta_{\rm D} = \left\{ (0,\phi) : \phi \in \Hm \right\}
\end{align}
or
\begin{align}\label{betaN}
	\beta_{\rm N} = \left\{ (x,0) : x \in \Hp \right\}.
\end{align}
Note that $\mu(\lambda)$ intersects $\beta_{\rm D}$ nontrivially whenever there is a solution to $Lu = \lambda u$ satisfying Dirichlet boundary conditions. Similarly, the subspace $\beta_{\rm N}$ encodes Neumann boundary conditions. 

Let $\cL$ denote the selfadjoint operator corresponding to the bilinear form $D$ in \eqref{Ddef}, with form domain $X = H^1_0(\Omega)$ or $X = H^2(\Omega)$ (for the Dirichlet and Neumann problems, respectively). The subspaces $\mu(\lambda)$ and $\beta$ comprise a Fredholm pair for each value of $\lambda$, so the Maslov index of $\mu$ with respect to $\beta$ is well defined, and satisfies
\begin{align}
	\Mas(\mu(\cdot); \beta) = - n(\cL),
\end{align}
where $n(\cL)$ denotes the number of strictly negative eigenvalues (i.e. the Morse index) of $\cL$.

We now turn to the constrained problem. We first require an assumption on the constrained space where the problem will be formulated. To state this assumption, we let $\gamma\colon H^1(\Omega) \to \Hp$ denote the Sobolev trace map.

\begin{hyp}
\label{hyp:bdry}
$L^2_c(\Omega) \subset L^2(\Omega)$ is a closed subspace such that
\begin{enumerate}
	\item $\gamma \left(H^1(\Omega) \cap L^2_c(\Omega) \right) = \Hp$;
	\item $\overline{H^1_0(\Omega) \cap L^2_c(\Omega)} = L^2_c(\Omega)$;
	\item $L^2_c(\Omega)^\perp$ is continuously embedded in $H^1(\Omega)$.
\end{enumerate}
\end{hyp}

Since $\gamma$ is surjective, it satisfies $\gamma \left(H^1(\Omega)\right) = \Hp$. Part (i) of the hypothesis prevents $L^2_c(\Omega)$ from being too small, and guarantees that the space of Cauchy data is rich enough to fully describe the spectral problem. The density condition (ii) ensures there are enough ``test functions" in $H^1_0(\Omega) \cap L^2_c(\Omega)$ to make sense of the constrained eigenvalue problem. The embedding condition (iii) means that $L^2_c(\Omega)^\perp \subset H^1(\Omega)$, and there is a constant $C>0$ so that
\begin{align}\label{embed}
	\|\phi\|_{H^1(\Omega)} \leq C \|\phi\|_{L^2(\Omega)}
\end{align}
for all $\phi \in L^2_c(\Omega)^\perp$. This condition implies that a weak solution $u$ to the constrained eigenvalue problem satisfies $Lu \in L^2(\Omega)$, hence $u \in H^2_{\rm loc}(\Omega)$. In Section \ref{sec:codim} we show that these conditions are always satisfied when $L^2_c(\Omega)^\perp$ is a finite-dimensional subspace of $H^1(\Omega)$.

Now consider the bilinear form \eqref{Ddef} restricted to $X \cap L^2_c(\Omega)$, where $X$ is the form domain of the unconstrained operator $\cL$. This defines a selfadjoint operator $\cL_c$, with dense domain $D(\cL_c) \subset  L^2_c(\Omega)$. This is the constrained operator whose spectrum we want to compute.
We define the space of Cauchy data for the constrained problem by
\begin{align*}
	\mu_c (\lambda) = \left\{ \left.\left( u, \frac{\p u}{\p \nu} \right)\right|_{\pO} : u \in H^1(\Omega) \cap L^2_c(\Omega) \text{ and } D(u,v) = \lambda \left<u,v\right> \text{ for all } v \in H^1_0(\Omega) \cap L^2_c(\Omega) \right\}.
\end{align*}

\begin{theorem}\label{thm:cMas}
If Hypotheses \ref{hyp:domain} and \ref{hyp:bdry} are satisfied, then $\mu_c(\lambda)$ has a well-defined Maslov index with respect to $\beta$, and there exists $\lambda_\infty < 0$ such that
\[
	n(\cL_c) = - \Mas \left( \mu_c\Big|_{[\lambda_\infty,0]}; \beta \right).
\]
\end{theorem}
In other words, the Maslov index computes the Morse index of the constrained operator $\cL_c$.

The classical approach to the constrained eigenvalue problem (see \cite{KP13,P11} and references therein) is to relate $n(\cL)$ and $n(\cL_c)$ through the index of a finite-dimensional ``constraint matrix."  
%

\begin{theorem}[\cite{CPV05}]
\label{thm:constraint}
Suppose $L^2_c(\Omega)$ has finite codimension, with $L^2_c(\Omega)^\perp = \operatorname{span} \{\phi_1, \ldots, \phi_m\}$. The constrained and unconstrained Morse indices are related by
\[
	n(\cL) - n(\cL_c) = \lim_{\lambda \to 0^-} n( M(\lambda) ).
\]
where $M(\lambda)$ is the $m\times m$ matrix with entries
\[
	M_{ij} (\lambda) = \left< (\cL -\lambda)^{-1}\phi_i, \phi_j\right>.
\]
\end{theorem}

\begin{rem}
A similar result appears in \cite{KP13}, with the added assumption that $\ker \cL \subset L^2_c(\Omega)$. This implies $\phi_i \in L^2_c(\Omega)^\perp \subset (\ker \cL)^\perp = \operatorname{ran} \cL$, so $M(0) =\left<\cL^{-1} \phi_i, \phi_j \right>$ is defined and the result simplifies to $n(\cL) - n(\cL_c) = n(M(0))$. The observation in \cite{P11} is that the limit of the Morse index of $M(\lambda)$ still exists without this assumption, even though some eigenvalues may diverge to $\pm\infty$.
\end{rem}

This result allows one to compute $n(\cL_c)$ from the unconstrained Morse index $n(\cL)$ and the constraint matrix $M$. Here we take a different approach, combining Theorem \ref{thm:cMas} with a homotopy argument to compute the constrained Morse index directly, without having to first know the unconstrained index.

To do this we describe what happens when the domain $\Omega$ is shrunk to a point through a smooth one-parameter family $\{\Omega_t\}$. The result is a constrained analog of Smale's Morse index theorem \cite{S65}, relating the Morse index of the operator to the number of conjugate points. Smale's result, which only applies to the Dirichlet problem, was originally proved by variational methods (see also \cite{U73}). A proof using the Maslov index was given in \cite{DJ11} for star-shaped domains, and in \cite{CJM14} for the general case.

We prove a general result to this effect in Section \ref{sec:MM}; for now we just state the simplest case, when Dirichlet boundary conditions are imposed and there is only one constraint function, i.e. $L^2_c(\Omega) = \{\phi\}^\perp$. We say that $t$ is a \emph{constrained conjugate point} for the Dirichlet problem if there exists a nonzero function $u \in H^2(\Omega_t) \cap H^1_0(\Omega_t)$ such that
\[
	\int_{\Omega_t} u\phi = 0, \quad L u = a\phi \text{ on } \Omega_t
\]
for some constant $a$. In other words, $0$ is an eigenvalue for the constrained Dirichlet problem on $\Omega_t$. Let $d(t)$ denote its multiplicity, so that $d(t) > 0$ whenever $t$ is a conjugate time.

The result is particularly simple when we assume that $\Omega_t$ shrinks to a point.

\begin{theorem}\label{thm:Morse}
Let $\{\Omega_t : 0 < t \leq 1\}$ be a smooth, increasing family of domains in $\bbR^n$, with $\Omega_1 = \Omega$. Suppose $L^2_c(\Omega) = \{\phi\}^\perp$ for some $\phi \in H^1(\Omega)$. If $|\Omega_t| \to 0$ as $t\to0$, then
\[
	n(\cL_c) = \sum_{t < 1} d(t).
\]
\end{theorem}

That is, the Morse index of the constrained operator equals the number of constrained conjugate points in $(0,1)$, counting multiplicity. The sum on the right-hand side is well defined because $d(t)$ is only nonzero for a finite set of times.

We conclude in Section \ref{sec:NLS} by giving a formal application of Theorem \ref{thm:Morse} to the ground state solution $\phi$ of the one-dimensional NLS. We find that there is a constrained conjugate point (hence a negative eigenvalue) if and only if the quantity
\[
	\frac{\partial}{\p \omega} \int_{-\infty}^\infty \phi^2
\]
is positive. This is the well-known Vakhitov--Kolokolov condition \cite{VK73}; see also \cite{GSS1}.

\subsection*{Acknowledgments} 
The authors would like to thank Yuri Latushkin and Dmitry Pelinovsky for helpful discussions during the preparation of this manuscript. G.C. was supported by an NSERC Discovery Grant. J.L.M. was supported in part by NSF Applied Math Grant DMS-1312874 and NSF CAREER Grant DMS-1352353.

\section{A finite-dimensional example}

We now give a simple illustration of Theorem \ref{thm:Morse}, by computing the constrained Morse index of $L = -\Delta - C$ on $[-1,1]$, where $C$ is a positive constant. We do this in three different ways: first by direct computation, and then using Theorems \ref{thm:constraint} and \ref{thm:Morse}.

Let $\cL$ denote the differential operator on $[-1,1]$ with Dirichlet boundary conditions, and $\cL_c$ the constrained operator on the space of zero mean functions
\[
	L^2_c(-1,1) = \left\{u \in L^2(-1,1) : \int_{-1}^1 u(x)\,dx = 0 \right\}.
\]
The constrained eigenvalue equation $\cL_c u = \lambda u$ is equivalent to the conditions
\[
	u_{xx} + Cu + \lambda u = \text{constant}, \quad \int_{-1}^1 u(x)\,dx = 0, \quad u(-1) = u(1) = 0.
\]
From the differential equation and the zero mean condition we obtain the general solution
\[
	u(x) = A(\cos \gamma x - \gamma^{-1} \sin\gamma) + B \sin \gamma x
\]
where $\gamma = \sqrt{C+\lambda}$. Imposing the Dirichlet boundary conditions at $x = \pm 1$, we have
\[
	A(\cos \gamma - \gamma^{-1} \sin\gamma) \pm B \sin \gamma = 0,
\]
which implies either $\cos\gamma = \gamma^{-1} \sin\gamma$ or $\sin\gamma=0$. Finally, observing that $\lambda<0$ iff $\gamma < \sqrt C$, we find that the number of negative eigenvalues is
\begin{align}\label{eq:MorseODE}
	n(\cL_c) = \#\left\{\gamma \in (0,\sqrt C) : \sin\gamma = 0 \text{ or } \tan\gamma = \gamma\right\}.
\end{align}

We next compute the Morse index using Theorem \ref{thm:Morse}, counting the number of conjugate points $t \in (0,1)$ for the family of domains $\Omega_t = (-t,t)$. The constrained equation on $\Omega_t$ is 
\[
	u_{xx} + Cu = \text{constant}, \quad \int_{-t}^t u(x)\,dx = 0.
\]
Setting $\gamma = \sqrt C$, we can write the general solution as
\[
	u(x) = A(\cos \gamma x - \gamma^{-1} \sin\gamma) + B \sin \gamma x.
\]
Therefore, $t \in (0,1)$ is a conjugate time precisely when
\[
	A\left( \cos \gamma t - (\gamma t)^{-1} \sin\gamma t \right) \pm B \sin \gamma t = 0.
\]
It follows that either $ \cos \gamma t = (\gamma t)^{-1} \sin\gamma t$ or $\sin \gamma t = 0$. Recalling that $\gamma = \sqrt C$, we obtain
\begin{align}
	\#\, \text{conjugate points} &= \# \left\{t \in (0,1) : \sin \sqrt Ct = 0 \text{ or } \tan \sqrt Ct = \sqrt Ct \right\},
\end{align}
which agrees with the Morse index computed in \eqref{eq:MorseODE}.

A similar computation shows that the unconstrained Morse index is
\[
	n(\cL) = \#\left\{\gamma \in (0,\sqrt C) : \sin\gamma = 0 \text{ or } \cos\gamma = 0 \right\}.
\]
Comparing solutions of $\cos\gamma=0$ and $\tan\gamma = \gamma$, we see that the constrained and unconstrained indices are related by
\begin{align}\label{eq:comparison}
	n(\cL) = \begin{cases} n(\cL_c) &\text{if } \tan \sqrt C> \sqrt C \\
	n(\cL_c) + 1 & \text{if } \tan \sqrt C \leq \sqrt C.
	\end{cases}
\end{align}

Finally, we verify that this is consistent with the prediction of Theorem \ref{thm:constraint} by computing the constraint matrix $M$. Since $L^2_c(-1,1)^\perp$ is spanned by the constant function $1$, $M$ is simply the  $1\times1$ matrix $\left<L^{-1} 1, 1\right>$. To compute $L^{-1}1$ we must solve the boundary value problem
\[
	u_{xx} + Cu + 1 = 0, \quad u(-1) = u(1) = 0.
\]
Again setting $\gamma = \sqrt C$, we find
\[
	u(x) = \frac{1}{\gamma^2} \left( \frac{\cos \gamma x}{\cos\gamma} - 1\right).
\]
and so
\begin{align*}
	\left<L^{-1}1,1\right> = \int_{-1}^1 u(x)\,dx 
	= \frac{2}{\gamma^2} \left( \frac{\tan\gamma}{\gamma} - 1 \right).
\end{align*}
Therefore $M$ will be nonpositive if and only if $\tan \sqrt C \leq \sqrt C$, as expected from comparing the result in \eqref{eq:comparison} with Theorem \ref{thm:constraint}.

\section{The constrained Maslov index}\label{sec:cMaslov}
In this section we define the Maslov index for constrained eigenvalue problems in multiple dimensions. After reviewing the Fredholm--Lagrangian Grassmannian and the Maslov index, as well as some necessary details of constrained operators and boundary value problems, we define the constrained Maslov index, and prove that it equals (minus) the constrained Morse index, thus proving Theorem \ref{thm:cMas}. As is common for such problems, most of the work goes into establishing the existence and regularity of the relevant paths of Lagrangian subspaces. Once this is known, the main result follows from a straightforward crossing form calculation.

Throughout the section we assume Hypotheses \ref{hyp:domain} and \ref{hyp:bdry}.

\subsection{The Maslov index in infinite dimensions}\label{sec:Maslov}
Before describing the constrained eigenvalue problem, we will review the infinite-dimensional Maslov index, following \cite{F04}.

Suppose $\cH$ is a symplectic Hilbert space: that is, a real Hilbert space equipped with a nondegenerate, skew-symmetric bilinear form $\omega$. A subspace $\mu \subset \cH$ is said to be isotropic if $\omega(v,w) = 0$ for all $v, w \in \mu$, and is said to be Lagrangian if it is isotropic and maximal, in the sense that it is not properly contained in any other isotropic subspace. The set of all Lagrangian subspaces is called the Lagrangian Grassmannian, and is denoted $\Lambda(\cH)$. This is a smooth, contractible Banach manifold, whose differentiable structure comes from associating to each Lagrangian subspace its orthogonal projection operator. Thus a family of Lagrangian subspaces $\mu(t)$ is of class $C^k$ if and only if the corresponding family of projections $P_{\mu(t)}$ is $C^k$.

We assume that the symplectic form can be written as $\omega(v,w) = \left<Jv,w\right>$, where $J \colon \cH \to \cH$ is a skew-symmetric operator satisfying $J^2 = -I$. If $\mu$ is a given Lagrangian subspace, and $A \colon \mu \to \mu$ is a bounded, selfadjoint operator, then the graph
\[
	\operatorname{Gr}_\mu(A) = \{v + JAv : v \in \mu\}
\]
will also be Lagrangian. Moreover, the orthogonal projection onto this graph can be computed algebraically from $A$; see \cite[Equation (2.16)]{F04}. Therefore, if $A(t)$ is a $C^k$ family of bounded, selfadjoint operators on $\mu$, the corresponding family of Lagrangian subspaces $\operatorname{Gr}_{\mu(t)}(A)$ will also be of class $C^k$. This simple observation is our main technical tool for establishing regularity properties of paths of Lagrangian subspaces.

Since $\Lambda(\cH)$ is contractible, there is no nontrivial notion of winding for general curves of Lagrangian subspaces, so we must restrict our attention to a smaller space in order to have a useful index theory. For a fixed Lagrangian subspace $\beta \subset \cH$, we define the Fredholm--Lagrangian Grassmannian,
\[
	\cF \Lambda_\beta(\cH) = \{\mu \in \Lambda(\cH) : \mu \text{ and } \beta \text{ are a Fredholm pair} \},
\]
recalling that $\mu$ and $\beta$ are said to be a Fredholm pair when $\mu \cap \beta$ is finite dimensional and $\mu + \beta$ is closed and has finite codimension. The Fredholm--Lagrangian Grassmannian is a smooth Banach manifold, with fundamental group $\pi_1(\cF \Lambda_\beta(\cH) ) = \bbZ$. Thus one can define an integer-valued homotopy invariant, the Maslov index, which measures the winding of a continuous path of Lagrangian subspaces $\mu(t)$, provided it remains Fredholm with respect to $\beta$ at all times. The utility of this index in eigenvalue problems stems from the fact that it is simply a count (with sign and multiplicity) of the nontrivial intersections between $\mu(t)$ and $\beta$.

To compute the Maslov index in practice, we use crossing forms. Suppose $\mu \colon [a,b] \to \cF \Lambda_\beta(\cH)$ is a continuously differentiable path of Lagrangian subspaces, and $\mu(t_*) \cap \beta \neq \{0\}$ for some $t_* \in [a,b]$. Let $v(t)$ be a continuously differentiable path in $\cH$, with $v(t) \in \mu(t)$ for $t$ close to $t_*$ and $v(t_*) \in \mu(t_*) \cap \beta$. The crossing form is a quadratic form defined on the finite-dimensional vector space $\mu(t_*) \cap \beta$ by
\[
	Q(v(t_*)) = \left. \omega \left(v, \frac{d v}{dt} \right)\right|_{t=t_*}.
\]
It can be shown that this depends only on the vector $v(t_*)$, and not on the path $v(t)$. If $Q$ is nondegenerate, then the crossing time $t_*$ is isolated. Suppose that $t_*$ is the only crossing in $[a,b]$ and let  $(n_+, n_-)$ be the signature of $Q$. The Maslov is then given by
\[
	\Mas\big(\mu_{[a,b]}; \beta \big) = \begin{cases} -n_- & \text{if } t_* = a,  \\n_+ - n_- & \text{if } t_* \in (a,b), \\ n_+ & \text{if } t_* = b. \end{cases}
\]
The Maslov index is additive, in the sense that
\[
	\Mas\big(\mu_{[a,b]}; \beta \big) = \Mas\big(\mu_{[a,c]}; \beta \big) + \Mas\big(\mu_{[c,b]}; \beta \big)
\]
for any $c \in (a,b)$, so we can use the crossing form to compute the Maslov index of any continuously differentiable curve, provided all of its crossings are nondegenerate.

If $E$ is a real Hilbert space, with dual space $E^*$, then $\cH = E \oplus E^*$ is a symplectic Hilbert space. The symplectic form is given by
\[
	\omega\big((x,\phi), (y,\psi) \big) = \psi(x) - \phi(y),
\]
and the corresponding complex structure $J \colon \cH \to \cH$ is
\[
	J(x,\phi) = \big( R^{-1}\phi, -Rx \big),
\]
where $R\colon E \to E^*$ is the isomorphism from the Riesz representation theorem.

To study selfadjoint boundary value problems we will take $E = \Hp$, hence $E^* = \Hm$. Elements of $\cH = \Hp \oplus \Hm$ will arise as the boundary values (or ``traces") of weak solutions to the eigenvalue equation $Lu = \lambda u$, or its constrained analogue, via the trace map
\begin{align}\label{eq:trdef}
	\tr u := \left.\left( u, \frac{\p u}{\p \nu} \right) \right|_{\pO}.
\end{align}
We will use integral notation to denote the dual pairing between $\Hp$ and $\Hm$, so Green's second identity yields
\begin{align}\label{eq:Green}
	\omega(\tr u, \tr v) = \int_\pO \left(u \frac{\p v}{\p \nu} - v \frac{\p u}{\p \nu}\right) = \int_\Omega ( u \Delta v - v \Delta u).
\end{align}
This identity hints at a connection between the Lagrangian subspaces of $\Hp \oplus \Hm$ and selfadjoint, second-order differential operators on $L^2(\Omega)$. While the current paper utilizes a particular version of this  correspondence, it is in fact part of a deeper phenomenon, which has been investigated systematically in \cite{LS16}.

\subsection{Preliminaries on constrained boundary value problems}
Recall that $L$ denotes the formal differential operator $-\Delta + V$, $D$ is the bilinear form defined in \eqref{Ddef}, and $L^2_c(\Omega)$ is a closed subspace of $L^2(\Omega)$ satisfying Hypothesis~\ref{hyp:bdry}. To define the trace of a weak solution, as in \eqref{eq:trdef}, we need to know that its normal derivative is well defined. The statement and proof of the next result, a constrained version of Green's first identity, closely follow \cite[Lemma 4.3]{M00}.

\begin{lemma}\label{lem:trace}
Let $u \in H^1(\Omega) \cap L^2_c(\Omega)$, and suppose there exists $f \in L^2_c(\Omega)$ such that $D(u,v) = \left<f,v\right>$ for all $v \in H^1_0(\Omega) \cap L^2_c(\Omega)$. Then there is a unique function $g \in \Hm$ such that 
\[
	D(u,v) = \left<Lu,v\right> + \int_\pO g v
\]
for all $v \in H^1(\Omega)$, and $g$ satisfies the estimate
\[
	\|g\|_{\Hm} \leq C \left( \|u\|_{H^1(\Omega)} + \|f\|_{L^2(\Omega)} \right).
\]
\end{lemma}

\begin{proof}
By Hypothesis \ref{hyp:bdry}(i) the constrained Dirichlet trace map $$\gamma_c := \gamma \big|_{L^2_c(\Omega)} \colon H^1(\Omega) \cap L^2_c(\Omega) \to \Hp$$ is surjective, and hence has a bounded right inverse, $E \colon \Hp \to H^1(\Omega) \cap L^2_c(\Omega)$. Now $g \in \Hm = \Hp^*$ can be defined by its action on $h \in \Hp$:
\[
	g(h) = D(u,Eh) - \left<f, Eh \right>.
\]
From the boundedness of $D$ and $E$ we obtain
\[
	|g(h)| \leq C \left( \|u\|_{H^1(\Omega)} + \|f\|_{L^2(\Omega)} \right) \|h\|_{\Hp}
\]
for all $h \in \Hp$, and the desired estimate follows.

If two such functions exist, say $g_1$ and $g_2$, we would have
\[
	\int_\pO (g_1 - g_2) v = 0
\]
for all $v \in H^1(\Omega) \cap L^2_c(\Omega)$. Since $\gamma_c$ is surjective, this implies $g_1 - g_2 = 0$.
\end{proof}

When $u$ and $v$ are sufficiently smooth, it follows from the classical version of Green's first identity that
\[
	\int_\pO g v = D(u,v) - \left<Lu,v\right> = \int_\pO \frac{\p u}{\p \nu} v.
\]
That is, $g$ is just the normal derivative of $u$. Thus in general we will refer to the function $g \in \Hm$ defined by Lemma \ref{lem:trace} as the normal derivative of $u$.

Note that this lemma does not immediately follow from the aforementioned result in \cite{M00} because we do not know a priori that $Lu \in L^2(\Omega)$. However, using Lemma \ref{lem:trace}, we can prove a posteriori that this is the case.

\begin{lemma}\label{lem:reg}
If $u$ satisfies the conditions of Lemma \ref{lem:trace}, then $Lu \in L^2(\Omega)$ and $PLu = f$.
\end{lemma}

\begin{proof}
To prove the result we will construct a function $F \in L^2(\Omega)$ that satisfies
\begin{align}\label{eq:Green2}
	D(u,v) = \left<F,v\right> + \int_\pO g v
\end{align}
for all $v \in H^1(\Omega)$. Such an $F$ must satisfy $\left<F,v\right> = \left<f,v\right>$ for all $v \in H^1(\Omega) \cap L^2_c(\Omega)$, and hence for all $v \in L^2_c(\Omega)$, which means $F = f + \phi$ for some $\phi \in L^2_c(\Omega)^\perp$.

We first claim that
\[
	H^1(\Omega) = \big(H^1(\Omega) \cap L^2_c(\Omega) \big) \oplus L^2_c(\Omega)^\perp.
\]
This follows from writing $v = Pv + (I-P)v$. Since $(I-P)v \in L^2_c(\Omega)^\perp \subset H^1(\Omega)$, we also have $Pv = v - (I-P)v \in H^1(\Omega)$ as required.

Now decompose $v \in H^1(\Omega)$ accordingly as $v_1 + v_2$. Using Lemma \ref{lem:trace} we obtain
\[
	D(u,v) = D(u,v_1) + D(u,v_2) = \left<f,v_1\right> + \int_\pO g v_1 + D(u,v_2).
\]
Comparing this to the right-hand side of \eqref{eq:Green2},
\[
	\left<F,v\right> + \int_\pO g v = \left<f,v_1\right> + \left<\phi, v_2\right> + \int_\pO g v_1 +  \int_\pO g v_2,
\]
we see that $\phi$ must satisfy
\begin{align}\label{blf}
	\left<\phi,v_2\right> = D(u,v_2) - \int_\pO g v_2
\end{align}
for all $v_2 \in L^2_c(\Omega)^\perp$. The inequality \eqref{embed} from Hypothesis \ref{hyp:bdry}(iii) implies the right-hand side of \eqref{blf} is a bounded linear functional on $L^2_c(\Omega)^\perp$, so the existence of $\phi$ follows from the Riesz representation theorem. Setting $F = f + \phi$ completes the proof of \eqref{eq:Green2}. Then for any $v \in H^1_0(\Omega)$ we obtain
\[
	D(u,v) = \left<F,v\right>,
\]
hence $Lu = F \in L^2(\Omega)$ and $PLu = PF = f$ as was claimed.
\end{proof}

We next give a result on the solvability of a Robin-type boundary value problem that will be needed in the proof of Lemma \ref{lem:muLag}. Suppose $u \in H^1(\Omega) \cap L^2_c(\Omega)$ satisfies
\begin{align}\label{eq:Robin}
	D(u,v) = \lambda \left<u,v\right> + \zeta \int_\pO (Ru)v
\end{align}
for every $v \in H^1(\Omega) \cap L^2_c(\Omega)$, where $R \colon \Hp \to \Hm$ is the Riesz duality map and $\zeta \in \bbR$. It follows from Lemmas~\ref{lem:trace} and \ref{lem:reg} that $PLu = \lambda u$ and
\[
	\frac{\p u}{\p \nu} - \zeta Ru = 0,
\]
and so we refer to this as a constrained Robin-type problem. Note that this is not a traditional Robin boundary value problem, even in the absence of constraints, on account of the Riesz operator $R$ that appears in the boundary conditions.

%

\begin{lemma}
\label{constrainedbdry}
For any fixed $\lambda_0 \in \bbR$, there exists $\zeta_0 \in \bbR$ such that the constrained Robin-type boundary value problem \eqref{eq:Robin} is invertible.
\end{lemma}

In particular, this means the homogeneous problem only admits the zero solution, whereas the inhomogeneous problem
\begin{align}\label{eq:inRobin}
	PLu = \lambda_0 u, \quad \frac{\p u}{\p \nu} - \zeta_0 Ru = h
\end{align}
has a unique solution for each $h \in \Hm$. This construction is the key ingredient in the proof of Lemma \ref{lem:muLag}, where it will be used to write the constrained Cauchy data space $\mu_c(\lambda)$ as the graph of a selfadjoint operator on a fixed Lagrangian subspace.

\begin{proof}
We will in fact prove that $\zeta_0$ can be chosen arbitrarily close to $0$. Let $\cL_\zeta$ be the selfadjoint operator corresponding to the bilinear form
\[
	D_\zeta(u,v) =  D(u,v) - \lambda_0 \left<u,v\right> - \zeta \int_\pO (Ru)v
\]
on $X \cap L^2_c(\Omega)$. By construction, $u \in \ker \cL_\zeta$ if and only if $u$ solves the homogeneous problem
\[
	P(L - \lambda_0) u = 0, \quad \frac{\p u}{\p \nu} - \zeta Ru = 0.
\]
It follows immediately from the proof of Theorem 3.2 in \cite{Rohleder2014} that the ordered eigenvalues of $\cL_\zeta$ are strictly monotone with respect to $\zeta$. Therefore, if $\cL_0$ is not invertible, $\cL_\zeta$ will be for any $0 < |\zeta| \ll 1$. 
\end{proof}

Finally, we discuss the relation of the selfadjoint operator $\cL_c$ to the operator $P \cL \big|_{L^2_c(\Omega)}$ that arises in the stability literature. Recall that $\cL$ is the operator corresponding to the bilinear form $D$ with form domain $X \subset H^1(\Omega)$, whereas $\cL_c$ corresponds to the form $D$ restricted to $X \cap L^2_c(\Omega)$. While it is easily verified that $P \cL u = \cL_c u$ for all $u \in D(\cL) \cap L^2_c(\Omega)$, we require an additional hypothesis on the constraint space to conclude that the two operators are identical, ie. that $D(\cL_c) = D(\cL) \cap L^2_c(\Omega)$.

\begin{prop}
If $L^2_c(\Omega)^\perp$ is continuously embedded in $X$, then $\cL_c = P \cL \big|_{L^2_c(\Omega)}$.
\end{prop}

For the Dirichlet problem this requires $L^2_c(\Omega)^\perp \subset H^1_0(\Omega)$, so all of the constraint functions $\phi \in L^2_c(\Omega)^\perp$ must satisfy Dirichlet boundary conditions. For the Neumann problem we already have $L^2_c(\Omega)^\perp \subset X = H^1(\Omega)$ by Hypothesis \ref{hyp:bdry}(iii).

\begin{proof}
First, suppose $u \in D(\cL) \cap L^2_c(\Omega)$, with $\cL u = f \in L^2(\Omega)$. From the definition of $\cL$, this means $D(u,v) = \left<f,v\right>$ for all $v \in X$. Thus for any $v \in X \cap L^2_c(\Omega)$ we have
\[
	D(u,v) = \left<f,v\right> = \left<Pf,v\right>,
\]
hence $u \in D(\cL_c)$ and $\cL_c u = Pf = P \cL u$. It follows that $P \cL \big|_{L^2_c(\Omega)} \subset \cL_c$.

To prove the other direction, we first observe that the form domain satisfies
\[
	X = \big(X \cap L^2_c(\Omega)\big) \oplus L^2_c(\Omega)^\perp.
\]
To see this, we decompose $u \in X$ as $u = Pu + (I-P)u$ and note that $(I-P)u \in L^2_c(\Omega)^\perp \subset X$, hence $Pu = u - (I-P)u \in X$.

Now suppose $u \in D(\cL_c)$, with $\cL_c u = f \in L^2_c(\Omega)$, so $D(u,v) = \left<f,v\right>$ for all $v \in X \cap L^2_c(\Omega)$. To complete the proof we must show that $u \in D(\cL)$, which entails constructing a function $F \in L^2(\Omega)$ such that $D(u,v) = \left<F,v\right>$ for all $v \in X$. Such an $F$ would necessarily satisfy
\[
	\left<F,v\right> = D(u,v) = \left<f,v\right>
\]
for all $v \in X \cap L^2_c(\Omega)$, hence $\left<F-f,v\right> = 0$ for all $v \in L^2_c(\Omega)$, and so we seek $F$ in the form $F = f + \phi$ for some $\phi \in L^2_c(\Omega)^\perp$.

As noted above, we can decompose $v \in X$ as
\[
	v = v_1 + v_2 \in \big(X \cap L^2_c(\Omega)\big) \oplus L^2_c(\Omega)^\perp.
\]
Then $D(u,v) = \left<F,v\right>$ if and only if
\[
	D(u,v_1) + D(u,v_2) = \left<f + \phi, v_1 + v_2\right> = \left<f,v_1\right> + \left<\phi, v_2\right>.
\]
Since $D(u,v_1) = \left<f,v_1\right>$ for $v_1 \in X \cap L^2_c(\Omega)$, we require $\phi$ to satisfy $D(u,v_2) = \left<\phi, v_2\right>$ for all $v_2 \in L^2_c(\Omega)^\perp$. By the continuous embedding hypothesis, the functional $v_2 \mapsto D(u,v_2)$ is bounded on $L^2_c(\Omega)^\perp$, and so $\phi$ exists by the Riesz representation theorem. It follows that $u \in D(\cL) \cap L^2_c(\Omega)$, and $P \cL u = PF = f = \cL_c u$, hence $\cL_c \subset P \cL \big|_{L^2_c(\Omega)}$.
\end{proof}

\subsection{Construction of the Maslov index}
We now have all of the ingredients in place to define the constrained Maslov index, and prove that it equals (minus) the Morse index of the constrained operator $\cL_c$.

The space of weak solutions for the constrained problem, in the absence of boundary conditions, is
\begin{align}\label{def:Kc}
	K_c(\lambda) = \left\{u \in H^1(\Omega) \cap L^2_c(\Omega) : D(u,v) = \lambda \left<u,v\right> \text{ for all } v \in H^1_0(\Omega) \cap L^2_c(\Omega)  \right\},
\end{align}
where the bilinear form $D$ is defined in \eqref{Ddef}. Any $u \in K_c(\lambda)$ satisfies the hypotheses of Lemma \ref{lem:trace}, with $f = \lambda u$, and so the boundary trace
\[
	\tr u := \left.\left( u, \frac{\p u}{\p \nu} \right) \right|_{\pO}
\]
is a well-defined element of $\Hh$, cf. \cite{CJM14}, and
\begin{align}\label{def:muc}
	\mu_c(\lambda) = \{ \tr u : u \in K_c(\lambda) \}
\end{align}
defines a subspace of $\Hh$. In fact, from Lemma \ref{lem:reg} we have $u \in H^2_{\rm loc}(\Omega)$, and so it follows from a unique continuation argument (as in \cite{BR12}) that
\[
	\tr \colon K_c(\lambda) \to \Hp \oplus \Hm
\]
is injective.

\begin{lemma}
\label{lem:muLag}
$\lambda \mapsto \mu_c(\lambda)$ is a smooth family of Lagrangian subspaces in $\cH$.
\end{lemma}

\begin{proof}
We first prove that $\mu_c(\lambda)$ is isotropic. Let $u,v \in K_c(\lambda)$. Then
\begin{align*}
	\omega(\tr u, \tr v) &= \int_\pO \left( u \frac{\p v}{\p \nu} - v \frac{\p u}{\p \nu} \right) \\
	&= D(v,u) - \left<\lambda v, u\right> - D(u,v) + \left<\lambda u, v\right> \\
	&= 0
\end{align*}
because $D$ is symmetric.

We now use the strategy of \cite[Proposition 3.5]{CJM14} to prove that $\mu_c(\lambda)$ is Lagrangian and is smooth with respect to $\lambda$. The idea, as described in Section \ref{sec:Maslov}, is to realize each subspace $\mu_c(\lambda)$ as the graph of a bounded, selfadjoint operator $A(\lambda)$ on a fixed Lagrangian subspace. This will imply each subspace is in fact Lagrangian, and the family $\{\mu_c(\lambda)\}$ is as smooth with respect to $\lambda$ as the family $\{A(\lambda)\}$ is. The operator $A(\lambda)$ will be a constrained Robin-to-Robin map for the $L-\lambda$. (The Neumann-to-Dirichlet map suffices whenever it is defined, i.e. when the constrained operator with Neumann boundary conditions is invertible.) The main modification to the argument in \cite{CJM14} stems from using Lemma \ref{constrainedbdry} to find a Robin-type boundary condition for which the constrained operator is invertible.

Since smoothness is a local property, it will suffice to construct $A(\lambda)$ in a neighborhood of a fixed $\lambda_0$. By Lemma \ref{constrainedbdry} there exists $\zeta_0 \in \bbR$ so that the constrained boundary value problem \eqref{eq:inRobin} is invertible for $\lambda = \lambda_0$, and hence for any nearby $\lambda$. Using this fixed value of $\zeta_0$ we define the subspace
\[
	\rho = \{ (f,g) \in \cH : f + \zeta_0 R^{-1}g = 0\},
\]
By construction, for any $(f,g) \in \rho$ there is a unique weak solution $u \in H^1(\Omega) \cap L^2_c(\Omega)$ to \eqref{eq:inRobin}, with $h = g - \zeta_0 Rf \in \Hm$. From this solution $u$ we define
\[
	A(\lambda)(f,g) = J^{-1}\Big( u\big|_\pO - f, \zeta_0 R\big(u\big|_\pO-f\big) \Big).
\]
Since $ \big( u|_\pO - f, \zeta_0 R(u|_\pO-f) \big)$ is contained in the subspace $J\rho = \{(f,g) : g = \zeta_0 R f \}$, we have $A(\lambda) \colon \rho \to \rho$ as desired. The proof that $A(\lambda)$ is selfadjoint follows directly from Green's identity, as in \cite{CJM14}.
\end{proof}

We next consider the boundary conditions. Since we have assumed $\beta$ is either $\beta_{\rm D}$ or $\beta_{\rm N}$, the following result is immediate.

\begin{lemma} 
\label{lem:betaLag}
The boundary space $\beta \subset \cH$ is Lagrangian.
\end{lemma}

Next, we study the intersection properties of $\mu_c(\lambda)$ and $\beta$.

\begin{lemma}\label{lem:Fredholm}
For each $\lambda \in \bbR$, $\mu_c(\lambda)$ and $\beta$ comprise a Fredholm pair, with $\dim \mu_c(\lambda) \cap \beta = \dim \ker (\cL_c - \lambda)$.
\end{lemma}

\begin{proof}
We follow the proof of Lemma 4 in \cite{CJM2}. Letting $P_\beta$ denote the orthogonal projection onto the boundary subspace $\beta \subset \cH$, and $P_\beta^\perp = I - P_\beta$ the complementary projection, it suffices to prove an estimate of the form
\begin{align}\label{H1estimate}
	\|u\|_{H^1(\Omega)} \leq C \left( \|u\|_{L^2(\Omega)} + \left\| P_\beta^\perp (\tr u) \right\|_{\cH} \right)
\end{align}
for all $u \in K_c(\lambda)$. Since we have assumed that $\beta = \beta_{\rm D}$ or $\beta = \beta_{\rm N}$, as defined in \eqref{betaD} and \eqref{betaN}, the boundary term is either
\[
	\left\| P_{\beta}^\perp (\tr u) \right\|_\cH = \big\| \left.u\right|_{\pO} \big\|_{\Hp} 
	\quad \text{or} \quad
	\left\| P_{\beta}^\perp (\tr u) \right\|_\cH =  \left\| \left.\frac{\p u}{\p \nu} \right|_{\pO} \right\|_{\Hm}.
\]

Now suppose $u \in K_c(\lambda)$. Letting $v=u$ in Lemma \ref{lem:trace}, we obtain
\[
	\int_\Omega \left[|\nabla u|^2 + (V-\lambda)u^2 \right] = \int_\pO u \frac{\p u}{\p \nu}.
\]
The energy estimate \eqref{H1estimate} now follows from the arithmetic--geometric mean inequality, as in \cite{CJM2}.

Finally, the fact that $\dim \mu_c(\lambda) \cap \beta = \dim \ker (\cL_c - \lambda)$ follows from the definitions of both spaces and the fact that the trace map is injective.
\end{proof}

\begin{rem}
In general, an estimate of the form \eqref{H1estimate} only implies that $\mu_c(\lambda) + \beta$ is closed and $\mu_c(\lambda) \cap \beta$ is finite dimensional. However, since $\mu_c(\lambda)$ and $\beta$ are already known to be Lagrangian (by Lemmas \ref{lem:muLag} and \ref{lem:betaLag}), we have
\[
	(\mu_c(\lambda) + \beta)^\perp = \mu_c(\lambda)^\perp \cap \beta^\perp = J(\mu_c(\lambda)) \cap J\beta = J (\mu_c(\lambda) \cap \beta).
\]
Since $J$ is an isomorphism and $\mu_c(\lambda) + \beta$ is closed, this implies $\codim(\mu_c(\lambda) + \beta) = \dim(\mu_c(\lambda) \cap \beta) < \infty$, so $\mu_c(\lambda)$ and $\beta$ are indeed a Fredholm pair.
\end{rem}

Combining Lemmas \ref{lem:muLag}, \ref{lem:betaLag} and \ref{lem:Fredholm}, we see that $\mu_c(\lambda)$ is a smooth path in $\cF \Lambda_\beta(\cH)$, so its Maslov index is well defined. In the final lemma of this section we relate this Maslov index to the Morse index of the constrained operator $\cL_c$, thus completing the proof of Theorem \ref{thm:cMas}.

\begin{lemma}\label{lemma:mono}
There exists $\lambda_\infty < 0$ such that $\mu_c(\lambda) \cap \beta = \{0\}$ for all $\lambda \leq \lambda_\infty$, and
\[
	n(\cL_c) = - \Mas \left( \mu_c\Big|_{[\lambda_\infty,0]}; \beta \right).
\]

\end{lemma}

\begin{proof}
We first prove the existence of $\lambda_\infty$. Suppose $\mu_c(\lambda) \cap \beta \neq \{0\}$, so the constrained eigenvalue problem has a nontrivial solution. That is, there exists $u \in H^1(\Omega) \cap L^2_c(\Omega)$ satisfying $P(L-\lambda)u =0$, with either Dirichlet or Neumann boundary conditions. It follows that
\[
	\lambda \int_\Omega u^2 = \int_\Omega \left[ |\nabla u|^2 + V u^2 \right],
\]
so $\lambda \geq \inf V$. Therefore any $\lambda_\infty < \inf V$ will suffice.

We next claim that the path $\mu_c(\lambda)$ is negative definite, in the sense that is always passes through $\beta$ in the same direction. This means the Maslov index is equal to (minus) the number of intersections of $\mu_c(\lambda)$ with $\beta$, hence
\begin{align*}
	\Mas \left( \mu_c\Big|_{[\lambda_\infty,0]}; \beta \right) = - \sum_{\lambda_\infty \leq \lambda < 0} \dim (\mu_c(\lambda) \cap \beta)
	= - \sum_{\lambda < 0} \dim (\mu_c(\lambda) \cap \beta)
	= - n(\cL_c).
\end{align*}
The second equality follows from the fact that there are no intersections for $\lambda < \lambda_\infty$, and the third equality is just the definition of the Morse index.

It only remains to prove the claimed monotonicity of $\mu_c(\lambda)$. We do this using crossing forms, as described in Section \ref{sec:Maslov}.

Suppose $u(\lambda)$ is a smooth curve in $K_c(\lambda)$, so $D(u,v) = \lambda\left<u,v\right>$ for all $v \in H^1_0(\Omega) \cap L^2_c(\Omega)$, hence $D(u',v) = \left<\lambda u' + u,v\right>$, where $'$ denotes differentiation with respect to $\lambda$. It follows from Lemma \ref{lem:trace} that
\begin{align*}
	\omega\left( \tr u, \tr u' \right) 
	&= \int_\pO \left(u \frac{\p u'}{\p \nu} - u' \frac{\p u}{\p \nu} \right) \\
	&= \big( D(u',u) - \left<\lambda u' + u, u \right> \big) - \big( D(u,u') - \lambda \left<u, u'\right> \big) \\
	&= - \int_\Omega u^2
\end{align*}
and so the path is negative definite as claimed.
\end{proof}

\subsection{The case of finite codimension}\label{sec:codim}
Before proving the constrained Morse index theorem, we show that Hypothesis \ref{hyp:bdry} is always satisfied when $L^2_c(\Omega)^\perp$ is contained in $H^1(\Omega)$ and has finite dimension.


\begin{lemma}
\label{lem:app}
If $L^2_c(\Omega)^\perp = \spn\{\phi_1, \ldots, \phi_m\}$ for functions $\phi_i \in H^1(\Omega)$, then
\[
	\gamma \left(H^1(\Omega) \cap L^2_c(\Omega) \right) = \Hp.
\]
and
\[
	\overline{H^1_0(\Omega) \cap L^2_c(\Omega)} = L^2_c(\Omega).
\]
\end{lemma}

%

\begin{proof}
Let $\chi_\epsilon$ be a smooth cutoff function on $\Omega$ that vanishes on the boundary and satisfies $\chi_\epsilon(x) = 1$ whenever $\operatorname{dist}(x,\pO) > \epsilon$. We assume without loss of generality that the $\{\phi_i\}$ are orthonormal.

Suppose $f \in \Hp$ is given. Since $\gamma \colon H^1(\Omega) \to \Hp$ is surjective (see, for instance, Lemma 3.37 of \cite{M00}), there exists $u \in H^1(\Omega)$ with $\gamma u = f$. Now define
\[
	u_c = u + \chi_\epsilon \sum_{i=1}^m \alpha_i \phi_i
\]
with coefficients $\alpha_1, \ldots, \alpha_n$ to be determined. Since $\chi_\epsilon$ vanishes on the boundary, $u_c$ satisfies $\gamma u_c = \gamma u = f$. Moreover, $u_c \in L^2_c(\Omega)$ if and only if
\[
	\sum_{i=1}^m \alpha_i \int_\Omega \chi_\epsilon \phi_i \phi_j = - \int_\Omega u \phi_j
\]
for each $j$. This is a linear equation for the coefficients $\{\alpha_i\}$, and will have a solution if the matrix
\[
	M_{ij}(\epsilon) = \int_\Omega \chi_\epsilon \phi_i \phi_j
\]
is invertible. The dominated convergence theorem implies
\[
	\lim_{\epsilon\to0} \int_\Omega \chi_\epsilon \phi_i \phi_j = \int_\Omega \phi_i \phi_j = \delta_{ij},
\]
hence $M_{ij}(\epsilon)$ is invertible for sufficiently small $\epsilon$.

The second claim follows from a similar construction. Suppose $u \in L^2_c(\Omega)$, so there exists a sequence $(u_k)$ in $H^1_0(\Omega)$ with $u_k \to u$ in $L^2(\Omega)$. To obtain a function in $H^1_0(\Omega) \cap L^2_c(\Omega)$, we replace each $u_k$ by
\[
	\tilde u_k = u_k + \chi_\epsilon \sum_{i=1}^m \alpha_i^k \phi_i,
\]
where $(\alpha_i^k)$ solve the linear equation
\[
	\sum_{i=1}^m M_{ij}(\epsilon) \alpha_i^k = - \int_\Omega u_k \phi_j
\]
and $\epsilon$ is chosen small enough to ensure $M_{ij}(\epsilon)$ is invertible. The fact that $u_k \to u$ implies
\[
	\int_\Omega u_k \phi_j \longrightarrow \int_\Omega u \phi_j = 0,
\]
for each $j$, hence $\alpha_i^k \to 0$ as $k \to \infty$. It follows that $\tilde u_k - u_k \to 0$, and so $\tilde u_k \to u$.
\end{proof}


\section{The constrained Morse index theorem}\label{sec:MM}
Now consider a one-parameter family of domains $\{\Omega_t\}_{a \leq t \leq b}$ in $\bbR^n$. For simplicity we will assume that each $\Omega_t$ has smooth boundary, that the domains are varying smoothly in time, and that the domains are increasing, in the sense that $\Omega_s \subset \Omega_t$ for $s<t$. See \cite[\S 2.2]{CJM14} for a description of the nonsmooth case (in the unconstrained problem).

The idea is to define a Maslov index with respect to the $t$ parameter, then use a homotopy argument to relate this to the Maslov index defined in Section \ref{sec:cMaslov}, and hence to the Morse index of the constrained operator. There is some freedom in how one chooses the constraints on $\Omega_t$ in relation to the original constraints on $\Omega$. Our choice is motivated by the requirement that the resulting path be monotone in $t$, which is necessary for the proof of Theorem \ref{thm:Morse}.

\subsection{The general index theorem}
First, we must describe the domain of the constrained operator on $\Omega_t$. Let $E_t \colon L^2(\Omega_t) \to L^2(\Omega)$ denote the operator of extension by zero, and define
\begin{align}
	L^2_c(\Omega_t) = \left\{u \in L^2(\Omega_t) : E_t u \in L^2_c(\Omega) \right\}.
\end{align}
In other words, $L^2_c(\Omega_t)$ consists of function whose extension by zero satisfies the constraints on the larger domain $\Omega$. To motivate this, suppose $L^2_c(\Omega) = \{\phi\}^\perp$ for some function $\phi$. Then for any function $u \in L^2(\Omega_t)$ we have
\[
	u \in L^2_c(\Omega_t) \ \Longleftrightarrow \ E_t u \in L^2_c(\Omega) \ \Longleftrightarrow \ \int_\Omega (E_t u)\phi = 0 
	\ \Longleftrightarrow \ \int_{\Omega_t} u \phi = 0,
\]
and so $L^2_c(\Omega_t) = \{ \phi|_{\Omega_t}\}^\perp$. We then define $\cL^t_c$ to be the selfadjoint operator corresponding to the bilinear form \eqref{Ddef} with form domain $H^1(\Omega_t) \cap L^2_c(\Omega_t)$ (for the Neumann problem) or $H^1_0(\Omega_t) \cap L^2_c(\Omega_t)$ (for the Dirichlet problem). Our index theorem computes the spectral flow of the family $\{\cL_c^t\}$, i.e. the difference in Morse indices, $n(\cL_c^b) - n(\cL_c^a)$. To describe this, it is convenient to reformulate the problem in terms of a $t$-dependent family of bilinear forms on a fixed domain.

To that end, we define the bilinear form
\begin{align}\label{Dtdef}
	D_t(u,v) = \int_{\Omega_t} \left[ \nabla(u \circ \varphi_t^{-1}) \cdot \nabla(v \circ \varphi_t^{-1}) + V (u \circ \varphi_t^{-1}) (v \circ \varphi_t^{-1}) \right]
\end{align}
for $u,v \in H^1(\Omega)$, and define the subspace
\[
	L^2_{c,t}(\Omega) = \{u \circ \varphi_t : u \in L^2_c(\Omega_t)\} \subset L^2(\Omega).
\]
Suppose $\phi \in L^2_c(\Omega)^\perp$. Then for any $u \in L^2_{c,t}(\Omega)$ we have $u \circ \varphi_t^{-1} \in L^2_c(\Omega_t)$, hence
\[
	0 = \int_{\Omega_t} (u \circ \varphi_t^{-1}) \phi = \int_\Omega u (\phi \circ \varphi_t) \det(D \varphi_t).
\]
In other words, the rescaled constraint space in $L^2(\Omega)$ is
\begin{align}
	L^2_{c,t}(\Omega)^\perp = \left\{ \det(D \varphi_t)(\phi \circ \varphi_t)  : \phi \in L^2_c(\Omega)^\perp \right\}.
\end{align}
This explicit description of the rescaled constraint functions will be used below in the crossing form calculation for the Dirichlet problem.

There is a formal differential operator $L_t$, and a boundary operator $B_t$, so that a version of Green's first identity
\[
	D_t(u,v) = \left<L_t u, v\right> + \int_{\pO} (B_t u)v
\]
holds if we additionally assume that $L_t u \in L^2(\Omega)$. We thus define the space of weak solutions to the (rescaled) constrained problem
\[
	K_c(\lambda,t) = \{u \in H^1(\Omega) \cap L^2_{c,t}(\Omega) : D_t(u,v) = \lambda \left<u,v\right> \text{ for all } v \in H^1_0(\Omega) \cap L^2_{c,t}(\Omega) \},
\]
and the space of Cauchy data
\[
	\mu_c(\lambda,t) = \{\tr_t u : u \in K_c(\lambda,t) \},
\]
using the rescaled trace map
\[
	\tr_t u := \left.\left( u, B_t u \right) \right|_{\pO}.
\]

\begin{theorem}
Let $\{\Omega_t\}$ be a smooth increasing family of domains, defined for $a \leq t \leq b$. If the family of subspaces $H^1(\Omega) \cap L^2_{c,t}(\Omega)$ is smooth, then
\begin{align}\label{eq:MorseMaslov}
	n(\cL_c^a) - n(\cL_c^b) = \Mas\left(\mu_c(0,\cdot)\Big|_{[a,b]}; \beta \right).
\end{align}
\end{theorem}

In other words, the Maslov index computes the spectral flow of the constrained family $\{\cL^t_c\}$. The smoothness assumption means there is a smooth family of $H^1$-bounded operators, $T_t \colon H^1(\Omega) \cap L^2_c(\Omega) \to H^1(\Omega)$, with $\operatorname{ran}(T_t) = H^1(\Omega) \cap L^2_{c,t}(\Omega)$. 

\begin{proof}
The proof is a standard application of the homotopy invariance of the Maslov index. The space $K_c(\lambda,t)$ of weak solutions is defined in terms of the form $D_t$ on $H^1(\Omega) \cap L^2_{c,t}(\Omega)$. Using the smoothness assumption, this is equivalent to the form $D_t \circ T_t$ on the fixed ($t$-independent) domain $H^1(\Omega) \cap L^2_c(\Omega)$. Since $D_t \circ T_t$ is a smooth family of forms, we can use the theory developed in \cite{CJM14} (which is reviewed in the proof of Lemma \ref{lem:muLag}), to see that
\[
	\mu_c \colon [\lambda_\infty,0] \times [a,b] \longrightarrow \cF \Lambda_\beta(\cH)
\]
is a smooth two-parameter family of Lagrangian subspaces.

This means the image under $\mu_c$ of the boundary of $[\lambda_\infty,0] \times [a,b]$ is null-homotopic, hence its Maslov vanishes. Summing the four sides of the boundary with the appropriate orientation, we obtain
\[
	\Mas\left(\mu(\cdot,a)\big|_{[\lambda_\infty,0]}\right) + \Mas\left(\mu(0,\cdot)\big|_{[a,b]}\right) = \Mas\left(\mu(\lambda_\infty,\cdot)\big|_{[a,b]}\right) + \Mas\left(\mu(\cdot,b)\big|_{[\lambda_\infty,0]}\right).
\]
The monotonicity computation in Lemma \ref{lemma:mono} shows that
\[
	\Mas\left(\mu(\cdot,t)\big|_{[\lambda_\infty,0]}\right) = -n(\cL_c^t)
\]
for any $t \in [a,b]$, and the choice of $\lambda_\infty$ implies that
\[
	\Mas\left(\mu(\lambda_\infty,\cdot)\big|_{[a,b]}\right) = 0.
\]
This completes the proof.
\end{proof}

\subsection{The Dirichlet crossing form}
We now complete the proof of Theorem \ref{thm:Morse} by computing the right-hand side of \eqref{eq:MorseMaslov} when $\beta$ is the Dirichlet subspace. This closely follows the crossing form computation in \cite[\S 5]{CJM14}. In particular, it suffices to prove that the crossing form is negative definite at any crossing.

Let $\{u_t\}$ be a smooth family of solutions to the constrained problem with $\lambda=0$, i.e. $\tr u_t \in K_c(0,t)$. This means $\int_\Omega u_t \phi_t = 0$ and $L_t u_t \propto \phi_t$, where
\begin{align}\label{def:phit}
	\phi_t = \det(D \varphi_t) (\phi \circ \varphi_t)
\end{align}
is the rescaled constraint function. More concretely, we can write $L_t u_t = a_t \phi_t$, where $a_t$ depends only on $t$, so that
\[
	D_t(u_t,v) = a_t \left<\phi_t , v\right> + \int_{\pO} (B_t u_t) v
\]
for all $v \in H^1(\Omega)$.

Letting $v = u_t'$, we obtain
\[
	D_t(u_t,u_t') = a_t \left<\phi_t , u_t'\right> + \int_{\pO} (B_t u_t) u_t'.
\]
On the other hand, differentiating with respect to $t$ and then plugging in $v = u_t$, we obtain
\[
	D_t'(u_t,u_t) + D_t(u_t',u_t) = a_t' \left<\phi_t , u_t\right> + a_t \left<\phi_t' , u_t\right> + \int_{\pO} (B_t u_t)' u_t.
\]
Therefore, the crossing form is
\begin{align}\label{Dcf1}
	Q(\tr_t u_t) = \omega( (\tr_t u_t), (\tr_t u_t)') = \int_{\pO} (B_t u_t)' u_t - (B_t u_t) u_t' 
	= D_t'(u_t,u_t) - 2 a_t \left<\phi_t' , u_t\right>,
\end{align}
where we have used the fact that $\left<u_t,\phi_t\right> = 0$, hence $\left<\phi_t, u_t'\right> = - \left<\phi_t', u_t\right>$.

To complete the computation we must find $D_t'$. Differentiating $D_t(u,u)$ for a fixed $u \in H^1(\Omega)$, we have
\begin{align}\label{Dtprime}
	D_t'(u,u) =  - 2D_t \left(u, \nabla_X (u \circ \varphi_t^{-1}) \circ \varphi_t \right) + \int_{\p\Omega_t} \left[ \left| \nabla(u \circ \varphi_t^{-1})\right|^2 + V \left|u \circ \varphi_t^{-1} \right|^2 \right] (X \cdot \nu_t).
\end{align}
Here, $X = \varphi_t'$ and hence we have used the fact that
\[
	\frac{d}{dt} (u \circ \varphi_t^{-1}) = -\nabla_X (u \circ \varphi_t^{-1})
\]
obtained by writing
\[
	0 = \frac{d}{dt} (u \circ \varphi_t^{-1} \circ \varphi_t) = \frac{d}{dt} (u \circ \varphi_t^{-1}) \circ \varphi_t + \nabla_X(u \circ \varphi_t^{-1}) \circ \varphi_t.
\]
We now assume that $t$ is a crossing time, so $\tr_t u_t \in \beta$. Evaluating the first term on the right-hand side of \eqref{Dtprime} at $u = u_t \in H^1_0(\Omega)$, and defining $\widehat u = u_t \circ \varphi_t^{-1}$, we obtain
\begin{align*}
	D_t \left(u_t, \nabla_X (u_t \circ \varphi_t^{-1}) \circ \varphi_t \right)  &= D_t(u_t, (\nabla_X \widehat u) \circ \varphi_t) \\
	&= \left<L_t u_t, (\nabla_X \widehat u) \circ \varphi_t \right> + \int_{\pO} (B_t u_t) (\nabla_X \widehat u) \circ \varphi_t \\
	&= a_t \left<\phi_t, (\nabla_X \widehat u) \circ \varphi_t \right> + \int_{\pOt} \left(\frac{\p \widehat u}{\p \nu_t}\right)^2 (X \cdot \nu_t).
\end{align*}
Since $\widehat u$ vanishes on $\pOt$, the boundary term in \eqref{Dtprime} simplifies to
\begin{align*}
		\int_{\p\Omega_t} \left[ \left| \nabla \widehat u\right|^2 + V |\widehat u|^2 \right] (X \cdot \nu_t) = \int_{\pOt} \left(\frac{\p \widehat u}{\p \nu_t}\right)^2 (X \cdot \nu_t)
\end{align*}
and so
\begin{align}\label{Dtprime2}
	D_t'(u_t,u_t) = - 2 a_t \left<\phi_t, (\nabla_X \widehat u) \circ \varphi_t\right> - \int_{\pOt} \left(\frac{\p u}{\p \nu_t}\right)^2 (X \cdot \nu_t) \, d\mu_t .
\end{align}
Combining this with \eqref{Dcf1}, we find that
\begin{align}\label{Dcross1}
	Q(\tr_t u_t) = - 2 a_t \left<\phi_t, (\nabla_X \widehat u) \circ \varphi_t\right> - 2 a_t \left<\phi_t' , u_t\right> - \int_{\pOt} \left(\frac{\p \widehat u}{\p \nu_t}\right)^2 (X \cdot \nu_t)  .
\end{align}

This expression for the crossing form is generally valid, in the sense that it holds for any smooth family of constraint functions $\{\phi_t\}$. We now show that our choice of $\phi_t$ is such that the first two terms on the right-hand side cancel, resulting in a form that is sign definite.

Differentiating \eqref{def:phit}, we obtain
\begin{align*}
	\phi_t' = \det(D\varphi_t) \dv(X) (\phi \circ \varphi_t) + \det(D\varphi_t) (\nabla_X \phi) \circ \varphi_t = \det(D\varphi_t) \dv(\phi X) \circ \varphi_t.
\end{align*}
On the other hand, we can use the divergence theorem, together with the fact that $\widehat u$ vanishes on $\pOt$, to write
\begin{align*}
	\left<\phi_t, (\nabla_X \widehat u) \circ \varphi_t\right> &= \int_\Omega \det(D \varphi_t) (\phi \circ \varphi_t)  (\nabla_X \widehat u) \circ \varphi_t \\
	&= \int_{\Omega_t} \phi \nabla_X \widehat u \\
	&= - \int_{\Omega_t} \widehat u \dv(\phi X) \\
	&= - \int_\Omega u_t \det(D \varphi_t) \dv(\phi X) \circ \varphi_t  \\
	&= - \left<u_t, \phi_t' \right>.
\end{align*}
Thus the first two terms on the right-hand side of \eqref{Dcross1} cancel, and the crossing form simplifies to
\begin{align}\label{Dcross2}
	Q(\tr u_t) = - \int_{\pOt} \left(\frac{\p \widehat u}{\p \nu_t}\right)^2 (X \cdot \nu_t).
\end{align}

\begin{rem}
The above computation readily generalizes to any number of constraints. The constrained equation becomes $L_t u_t = \sum a_t^i \phi_t^i$, where $\{\phi_t^i\}$ are the rescaled constraint functions, and the right-hand side of \eqref{Dcross1} simply becomes
\[
 - 2 \sum a_t^i \left( \left<\phi^i_t, (\nabla_X \widehat u) \circ \varphi_t\right> + \left<(\phi^i_t)' , u_t\right> \right) - \int_{\pOt} \left(\frac{\p \widehat u}{\p \nu_t}\right)^2 (X \cdot \nu_t) = - \int_{\pOt} \left(\frac{\p \widehat u}{\p \nu_t}\right)^2 (X \cdot \nu_t),
\]
where each of the terms in brackets vanishes, as in the case of a single constraint.
\end{rem}


\section{An application to the Nonlinear Schr\"odinger equation}
\label{sec:NLS}
As a final application of Theorem \ref{thm:Morse}, we study the stability of the ground state solution of the nonlinear Schr\"odinger equation \eqref{eq:schr}. Under a mild condition on $f$ and $\omega$ (see, for instance, \cite[Appendix]{BL80}), there exists an even, positive solution $\phi$ to
\begin{align}\label{eq:ss}
	\phi_{xx} + f(\phi^2)\phi + \omega\phi = 0, \quad \lim_{x\to\infty} \phi(x) = 0
\end{align}
that is decreasing for $x>0$. It follows that $L_- \phi = 0$, hence $n(L_-) = 0$. We are interested in computing the Morse index of $L_+$, constrained to $(\ker L_-)^\perp = \{\phi\}^\perp$. We do this by applying Theorem \ref{thm:Morse} to the semi-infinite domain $\Omega_t = (-\infty,t)$.

\begin{rem}
The following computation is not, strictly speaking, an application of Theorem \ref{thm:Morse}, which was only proved for bounded domains. A Morse--Maslov index theorem for semi-infinite domains recently appeared in \cite{BCJLMS17}. Instead of focusing on technical details, we simply compute the conjugate points, and observe that a formal application of Theorem \ref{thm:Morse} yields the well-known stability criterion of Vakhitov and Kolokolov.
\end{rem}

Differentiating \eqref{eq:ss} with respect to $x$, we obtain $L_+(\phi_x) = 0$. Moreover, assuming the map $\omega \mapsto \phi$ is $C^1$, we also have $L_+ (\phi_\omega) = \phi$. The homogeneous equation $L_+u = 0$ has one linearly independent solution that decays as $x \to -\infty$, namely $\phi_x$, and so any solution to the inhomogeneous equation $L_+ u = \phi$ that decays at $-\infty$ is of the form
\[
	u = A \phi_x + \phi_\omega
\]
for some $A \in \bbR$. The constraint on $(-\infty,t)$ is
\begin{align}\label{eq:NLSconstraint}
	0 = \int_{-\infty}^t u\phi = \int_{-\infty}^t \left( A \phi \phi_x + \phi \phi_\omega \right) = \frac12 \left( A \phi^2(t) + \frac{\partial}{\partial\omega} \int_{-\infty}^t \phi^2(x)\, dx \right)
\end{align}
and the Dirichlet boundary condition at $t$ is
\begin{align}\label{eq:NLSboundary}
	A\phi_x(t) + \phi_\omega(t) = 0.
\end{align}
By definition, $t \in \bbR$ is a conjugate point when both \eqref{eq:NLSconstraint} and \eqref{eq:NLSboundary} are satisfied.

Now define
\[
	g(u) = u^{-1} \int_0^u f(v)\,dv, \quad u > 0,
\]
so that $u g'(u) + g(u) = f(u)$. It follows that $\phi_x^2 + [\omega + g(\phi^2)] \phi^2$ is constant for any solution to \eqref{eq:ss}. Since $\phi \in H^1(\bbR)$, we have $\phi_x^2 + [\omega + g(\phi^2)] \phi^2 = 0$. Differentiating this equation with respect to $\omega$, we find that
\begin{align}\label{eq:phix}
	\phi_x \phi_{x\omega} + g'(\phi^2) \phi^3 \phi_\omega + \frac12 \phi^2 + [\omega + g(\phi^2)] \phi \phi_\omega = 0.
\end{align}
Since $\phi$ is positive, this implies $\phi_x$ and $\phi_\omega$ do not simultaneously vanish. Together with \eqref{eq:NLSboundary}, this implies $\phi_x(t) \neq 0$ if $t$ is a conjugate point, hence $A = -\phi_\omega(t)/\phi_x(t)$. Substituting this into \eqref{eq:NLSconstraint}, we find that $t$ is a conjugate point precisely when it is a root of the function
\begin{align}
	c(t) = - \frac{\phi^2(t) \phi_\omega(t)}{\phi_x(t)} + \frac{\partial}{\partial\omega} \int_{-\infty}^t \phi^2(x)\,dx.
\end{align}
This function is plotted in Figure~\ref{fig:c} for the power law $f(\phi^2) = \phi^{2p}$. In this case it is easily verified that there is a conjugate point if and only if $p>2$.


\begin{figure}[!tbp]
	\begin{subfigure}[b]{0.3\textwidth}
		\includegraphics[width=\textwidth]{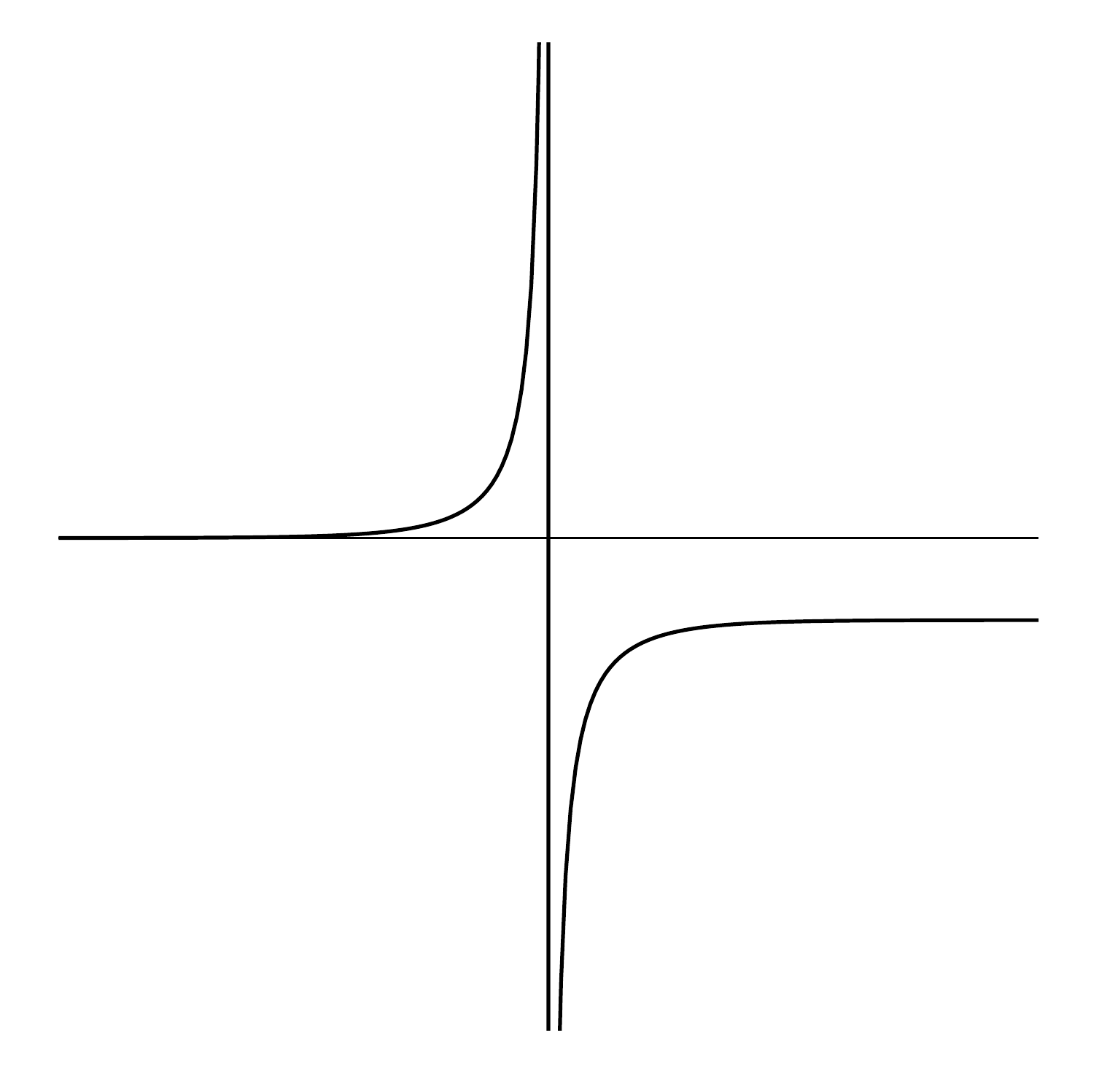}
		\caption{$p=1$}
	\end{subfigure}
	\begin{subfigure}[b]{0.3\textwidth}
		\includegraphics[width=\textwidth]{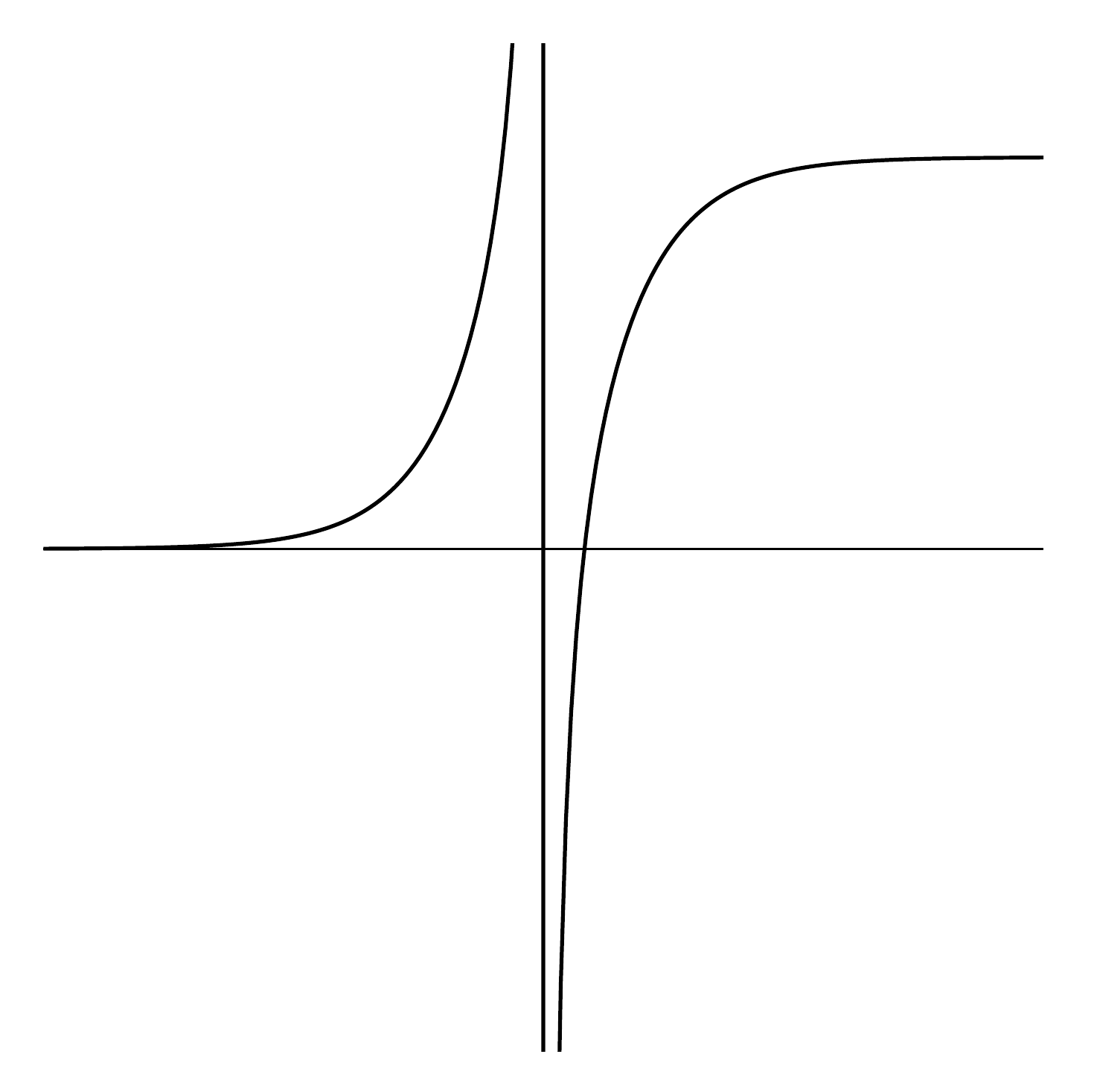}
		\caption{$p=3$}
	\end{subfigure}
\caption{The function $c(t)$ for the power-law nonlinearity $f(\phi^2) = (p+1)\phi^{2p}$.}
\label{fig:c}
\end{figure}

It is not difficult to verify that the function $c$ has the following properties:
\begin{enumerate}
\item $\displaystyle{\lim_{t\to -\infty} c(t) = 0}$;
\item $\displaystyle{\lim_{t\to 0-} c(t) = \infty}$;
\item $\displaystyle{\lim_{t\to 0+} c(t) = -\infty}$;
\item $\displaystyle{\lim_{t\to\infty} c(t) = \frac{\partial}{\partial\omega} \int_{-\infty}^\infty \phi^2(x)\, dx}$;
\item $c'(t) > 0$ for $t \neq 0$.
\end{enumerate}

In particular, for (v) we differentiate to obtain
\begin{align*}
	c'(t) = \phi^2 \left( \frac{\phi_\omega \phi_{xx} - \phi_{\omega x} \phi_x}{\phi_x^2} \right) \bigg|_{x=t}
\end{align*}
for $t \neq 0$. Using \eqref{eq:ss} and \eqref{eq:phix} to compute $\phi_{xx}$ and $\phi_{x\omega}$ we find that
\begin{align*}
	\phi_\omega \phi_{xx} - \phi_{\omega x} \phi_x 
	&= \frac12 \phi^2.
\end{align*}
Since $\phi>0$, this implies $c'(t) > 0$.

It follows immediately that there are no conjugate points in $(-\infty,0)$, and there is a single conjugate point in $(0,\infty)$ if and only if the ``slope" $\partial_\omega \int \phi^2$ is positive.


\bibliographystyle{amsplain}
\bibliography{maslov}
 
\end{document}